\documentclass[11pt]{article}
\usepackage{amsfonts,amsthm,amsmath}
\usepackage{url}
\newcommand{\zed}{\ensuremath{\mathbb Z}}
\newcommand{\eff}{\ensuremath{\mathbb F}}
\newtheorem{Theorem}{Theorem}[section]

\newtheorem{Example}{Example}[section]
\newtheorem{Remark}{Remark}[section]
\newtheorem{Lemma}[Theorem]{Lemma}
\newtheorem{Corollary}[Theorem]{Corollary}

\title {\textbf{Some results on the existence of $t$-all-or-nothing 
transforms over arbitrary alphabets}}
\author{Navid Nasr Esfahani, Ian Goldberg\thanks{Research supported by  NSERC discovery grant RGPIN-341529 } 
\ and Douglas~R.~Stinson\thanks{Research supported by  NSERC discovery grant RGPIN-03882}\\
David R.\ Cheriton School of Computer Science\\ University of Waterloo\\
Waterloo, Ontario N2L 3G1, Canada}
\date{\today}

\begin{document}
\maketitle

\begin{abstract}
A $(t, s, v)$-all-or-nothing transform is a bijective mapping defined
on $s$-tuples over an alphabet of size $v$, which satisfies the condition
that the values of any $t$ input co-ordinates are
completely undetermined, given only the values of any $s-t$ output co-ordinates.
The main question we address in this paper is:
for which choices of parameters does a $(t, s, v)$-all-or-nothing transform (AONT) exist?
More specifically, if we fix $t$ and $v$, we want to determine the maximum
integer $s$ such that a $(t, s, v)$-AONT exists. 
We mainly concentrate on the case $t=2$ for arbitrary values of $v$, where we 
obtain various necessary as well as sufficient conditions for existence of these objects.
We consider both linear and general (linear or nonlinear) AONT. 
We also show some connections between AONT, orthogonal arrays and resilient functions. 
\end{abstract}

\section{Introduction and Previous Results}
\label{intro.sec}

Rivest defined all-or-nothing transforms in \cite{R} in the setting of computational
security. Stinson considered unconditionally secure all-or-nothing transforms in \cite{St}. 
More general types of
unconditionally secure all-or-nothing transforms have been recently studied in \cite{DES,ES,ZZWG}.

We begin with some relevant definitions.
Let $X$ be a finite set of cardinality $v$, called an {\em alphabet}.
Let $s$ be a positive integer, and suppose that $\phi : X^s \rightarrow X^s$.
We will think of $\phi$ as a function that maps
an input $s$-tuple, say ${\mathbf x} = (x_1, \dots , x_s)$, to an
output $s$-tuple, say ${\mathbf y} = (y_1, \dots , y_s)$, where
$x_i,y_i \in X$ for $1 \leq i \leq s$.
Let $1 \leq t \leq s$ be an integer.
Informally, the function $\phi$ is an (unconditionally secure) 
{\em $(t, s,v)$-all-or-nothing transform} provided that 
the following properties are satisfied:
\begin{enumerate}
\item $\phi$ is a bijection.
\item If any $s-t$ of the $s$ output values
$y_1, \dots , y_s$ are
fixed, then the values of any $t$ inputs are
completely undetermined, in an information-theoretic sense.
\end{enumerate}
We will denote such a function as
a $(t, s, v)$-AONT, where $v = |X|$.

We note that any bijection from $X^s$ to itself is an $(s, s, v)$-AONT, so 
the case $s=t$ is trivial.

The work of Rivest \cite{R} and Stinson \cite{St} concerned the case $t=1$.
Rivest's original motivation for AONT involved block ciphers. The idea is
to apply a $(1, s, v)$-AONT to $s$ plaintext blocks, where each plaintext
block is treated as an element over an alphabet of size $v$. After the AONT is
applied the resulting $s$ blocks are then encrypted. The AONT property ensures
that all $s$ ciphertext blocks must be decrypted in order to obtain any information
about any single plaintext block.

Other applications of AONT are enumerated in \cite{DES}, where
AONT (and ``approximations'' to AONT) for $t \geq 2$ were first studied. 
The paper \cite{DES}  mainly considers the case
$t=v=2$. Additional results in this case are found in \cite{ZZWG} and \cite{ES}; 
the latter paper also contains some results for $t=2$, $v=3$.
In this paper, we study AONT for arbitrary values of $v$ and $t$, obtaining
our most detailed results for the case $t=2$.

The definition of AONT can be rephrased in terms of the entropy function  $\mathsf{H}$.  
Let 
\[\mathbf{X_1}, \dots , \mathbf{X_s}, \mathbf{Y_1}, \dots , \mathbf{Y_s}\] be random variables taking on
values in the finite set $X$.  
These $2s$ random variables 
define a $(t, s, v)$-AONT provided that
the following conditions are satisfied:
\begin{enumerate}
\item $\mathsf{H}( \mathbf{Y_1}, \dots , \mathbf{Y_s} \mid \mathbf{X_1}, \dots , \mathbf{X_s}) = 0$.
\item $\mathsf{H}( \mathbf{X_1}, \dots , \mathbf{X_s} \mid \mathbf{Y_1}, \dots , \mathbf{Y_s}) = 0$.
\item For all $\mathcal{X} \subseteq \{\mathbf{X_1}, \dots , \mathbf{X_s}\}$ with 
$|\mathcal{X}|  = t$, and for all $\mathcal{Y} \subseteq \{\mathbf{Y_1}, \dots , \mathbf{Y_s}\}$ with 
$|\mathcal{Y}|  = t$, it holds that 
 \begin{equation}
 \label{t-AONT.eq}
 \mathsf{H}( \mathcal{X}  \mid \{ \mathbf{Y_1}, \dots , \mathbf{Y_s}\}  \setminus \mathcal{Y} ) = \mathsf{H}(\mathcal{X}).
 \end{equation}
\end{enumerate}

Let $\eff_q$ be a finite field of order $q$.
An AONT with alphabet $\eff_q$ is {\em linear}
if each $y_i$ is an $\eff_q$-linear function of 
$x_1, \dots , x_s$. Then, we can write 
\begin{equation}
\label{linear.eq}
{\mathbf y} = \phi({\mathbf x}) = {\mathbf x}M^{-1}
\quad \text{and} \quad  {\mathbf x} = \phi^{-1}({\mathbf y}) = {\mathbf y}M,
\end{equation}
where $M$ is an invertible $s$ by $s$ matrix with entries from $\eff_q$. 
Subsequently, when we refer to a ``linear AONT'', we mean the matrix $M$ that
transforms ${\mathbf y}$ to ${\mathbf x}$, as specified in (\ref{linear.eq}).

The following lemma from \cite{DES} characterizes 
linear all-or-nothing transforms in terms of submatrices of the matrix $M$.

\begin{Lemma}
\cite[Lemma 1]{DES}
\label{linear}
Suppose that $q$ is a prime power and $M$ is 
an invertible $s$ by $s$ matrix with entries from $\eff_q$. 
Then $M$ defines a linear $(t, s, q)$-AONT if and only if every 
$t$ by $t$ submatrix of $M$ is invertible.
\end{Lemma}

\begin{Remark}
\label{ssq.rem}
Any invertible  $s$ by $s$ matrix with entries from $\eff_q$
defines a linear $(s, s, q)$-AONT.
\end{Remark}

An $s$ by $s$ \emph{Cauchy matrix} can be defined over $\eff_q$ 
if $q\geq 2s$. Let $a_1, \dots , a_s,b_1, \dots , b_s$ be distinct elements of
$\eff_q$. Let $c_{ij} = 1/(a_i-b_j)$, for $1 \leq i \leq s$ and $1 \leq j \leq s$.
Then $C = (c_{ij})$ is the Cauchy matrix defined by the sequence  $a_1, \dots , a_s,b_1, \dots , b_s$.
The most important property of a Cauchy matrix $C$ is that any square submatrix of $C$ (including $C$
itself) is invertible over $\eff_q$.

Cauchy matrices were briefly mentioned in \cite{St} as a possible method of 
constructing $1$-AONT. 
It was noted in \cite{DES} that, when $ q \geq 2s$,
Cauchy matrices immediately yield the strongest possible all-or-nothing transforms, 
as stated in the following theorem.

\begin{Theorem}
\cite[Theorem 2]{DES}
\label{cauchy}
Suppose $q$ is a prime power and $q \geq 2s$. Then there is a linear transform
that is simultaneously a $(t,s,q)$-AONT for all $t$ such that $1 \leq t \leq s$.
\end{Theorem}

We observe that, in general, the existence of a $(t,s,q)$-AONT
does not necessarily imply the existence of a $(t-1,s,q)$-AONT
or a $(t+1,s,q)$-AONT.

We next review some results on general (i.e., linear or nonlinear) AONT.
Let $A$ be an $N$ by $k$ array whose entries are elements chosen from 
an alphabet $X$ of size $v$.  
We will refer to $A$ as an {\em $(N,k,v)$-array}. 
Suppose the columns of $A$ are
labelled by the elements in the set $C = \{ 1, \dots , k\}$.  
Let $D \subseteq C$, and define $A_D$ to be the array obtained from $A$
by deleting all the columns $c \notin D$.
We say that $A$ is
{\em unbiased} with respect to $D$ if the rows of
$A_D$ contain every $|D|$-tuple of elements of $X$ 
exactly $N / v^{|D|}$ times.

The following result characterizes $(t,s,v)$-AONT in terms of
arrays that are unbiased with respect to
certain subsets of columns.

\begin{Theorem}
\cite[Theorem 34]{DES}
\label{equiv}
A $(t,s,v)$-AONT is equivalent to a $(v^s,2s,v)$-array
that is unbiased with respect to the following subsets of columns:
\begin{enumerate}
\item $\{1, \dots , s\}$,
\item $\{s+1, \dots , 2s\}$, and
\item $I \cup \{ s+1, \dots , 2s\} \setminus J$, 
         for all $I \subseteq \{1,\dots , s\}$ with $|I| = t$ and all 
         $J \subseteq \{s+1,\dots , 2s\}$ with $|J| = t$.
\end{enumerate}
\end{Theorem}

An OA$_{\lambda}(t,k,v)$ (an {\em orthogonal array})
is a $(\lambda v^t,k,v)$-array that is unbiased with respect to
any subset of $t$ columns. If $\lambda = 1$, then we simply write
the orthogonal array as an OA$(t,k,v)$.

The following corollary of Theorem \ref{equiv} is immediate.

\begin{Corollary}
\cite[Corollary 35]{DES}
If there exists an OA$(s,2s,v)$, then there exists a $(t,s,v)$-AONT
for all $t$ such that $1 \leq t \leq s$.
\end{Corollary}

For prime powers $q$, the existence of $(1,s,q)$-AONT has 
been completely determined in \cite{St}.

\begin{Theorem}
\cite[Corollary 2.3]{St}
There exists a linear  $(1,s,q)$-AONT for all prime powers $q > 2$ and for all positive integers $s$.
\end{Theorem}

When $q=2$, we have the following.

\begin{Theorem}
\cite[Theorem 3.5]{St}
There does not exist a  $(1,s,2)$-AONT for any integer $s> 1$.
\end{Theorem}

\subsection{Organization of the Paper}

Section \ref{linear.sec} deals with linear AONT. First, we give a construction
for certain $(2,s,q)$-AONT as well as a nonexistence result. 
In Section \ref{2qq.sec}, we focus on $(2,q,q)$-AONT and in Section \ref{computer.sec}
we report the results
of some enumerations of small cases. 
Section \ref{equivalence.sec} discusses the notion of equivalence
of linear AONT. Section \ref{lin-gen} examines linear
$(t,s,q)$-AONT and shows a connection with linear $t$-resilient functions. 
Section \ref{general.sec} shows some relations between 
general AONT, orthogonal arrays and resilient functions.
Finally, Section \ref{summary.sec} summarizes the paper and gives some open problems.

\section{New Results on Linear AONT}
\label{linear.sec}


We begin this section with a construction.

\begin{Theorem}
\label{vandermonde}
Suppose $q = 2^n$, $q-1$ is prime and $s \leq q-1$.
Then there exists a linear $(2,s,q)$-AONT over $\eff_q$.
\end{Theorem}

\begin{proof}
Let $\alpha \in \eff_q$ be a primitive element and let $M = (m_{r,c})$ be the $s$ by $s$ Vandermonde matrix 
in which $m_{r,c} = \alpha^{rc}$, $0 \leq r,c \leq s-1$. Clearly $M$ is invertible, so
we only need to show that any $2$ by $2$ submatrix is invertible. Consider a submatrix $M'$
defined by rows $i, j$ and columns $i',j'$, where
$i \neq j$ and $i' \neq j'$. We have 
\[ \det(M') = \alpha ^{ii' + jj'} - \alpha ^{ij' + ji'}, \]
so $\det(M') = 0$ if and only if $\alpha ^{ii' + jj'} = \alpha ^{ij' + ji'}$, which 
happens if and only if 
\[ ii' + jj' \equiv ij' + ji' \bmod (q-1).\]
This condition is equivalent to 
\[ (i-j)(i' - j') \equiv 0 \bmod (q-1).\]
Since $q-1$ is prime, this happens if and only if $i = i'$ or $j = j'$.
We assumed $i \neq j$ and $i' \neq j'$, so we conclude that $M'$ is invertible.
\end{proof}

The above result requires that $2^n-1$ is a (Mersenne) prime. Here are a couple of results on
Mersenne primes from \cite{Mersenne}. The first few Mersenne primes
occur for \[n = 2,3,5,7,13,31,61,89,107,127.\] At the time this paper was written, 
there were 49 known Mersenne primes, the largest being  $2^{74207281}-1$, 
which was discovered in January 2016.

\bigskip

If we ignore the requirement that a linear AONT is an invertible matrix,
then a construction for $q$ by $q$ matrices is easy.

\begin{Theorem} For any prime power $q$, there is a $q$ by $q$ matrix 
defined over $\eff_q$ such that any $2$ by $2$ submatrix is invertible.
\end{Theorem}

\begin{proof}
$M = (m_{r,c})$ be the $q$ by $q$ matrix of entries from $\eff_q$ defined by
the rule $m_{r,c} = r + c$, where the sum is computed in $\eff_q$.
Consider a submatrix $M'$
defined by rows $i, j$ and columns $i',j'$, where
$i \neq j'$ and $i' < j'$. We have 
\[ \det(M') = {ii' + jj'} - {(ij' + ji')}, \]
so $\det(M') = 0$ if and only if ${ii' + jj'} = {ij' + ji'}$.
This condition is equivalent to 
\[ (i-j)(i' - j') = 0,\]
which happens if and only if $i = i'$ or $j = j'$.
We assumed $i \neq j$ and $i' \neq j'$, so we conclude that $M'$ is invertible.
\end{proof}

We note that the above construction does not yield an AONT for $q > 2$, because
the sum of all the rows of the constructed matrix $M$ is the all-zero vector
and hence $M$ is not invertible.

We next define a ``standard form'' for linear AONT.
Suppose $M$ is a matrix for a linear $(2,s,q)$-AONT.
Clearly there can be at  most one zero in each row and column of $M$.
Then we can permute the rows and columns so that the 0's comprise the
first $\mu$ entries on the main diagonal of $M$. 
If $\mu = 0$, then 
we can multiply rows and columns by
nonzero field elements so that all the entries in the first row and first column 
consist of $1$'s.
If $\mu \neq 0$, we can multiply rows and columns by
nonzero field elements so that all the entries in the first row and first column 
consist of $1$'s, except for the entry in the top left corner, which is a $0$.
Such a matrix $M$ is said to be of \emph{type $\mu$ standard form}.

\begin{Theorem} 
\label{nonexistence}
There is no linear $(2,q+1,q)$-AONT for any prime power  $q > 2$.
\end{Theorem}
\begin{proof}
Suppose $M$ is a matrix for a linear $(2,q+1,q)$-AONT defined over $\eff_q$.
We can assume that $M$ is in standard form.
Consider the $q+1$ ordered pairs occurring in any two fixed rows of the matrix $M$. 
There are $q$ symbols, which result in $q^2$ possible ordered pairs.
However, the pair consisting of two zeros is ruled out, leaving $q^2-1$ ordered pairs. 
For two such ordered pairs $(i,j)^T$ and $(i',j')^T$, define
$(i,j)^T \sim (i',j')^T$ if there is a nonzero element
$\alpha \in \eff_q$ such that $(i,j)^T = \alpha (i',j')^T$.
Clearly $\sim$ is an equivalence relation, and there are
$q+1$ equivalence classes, each having size $q-1$. 
We can only have at most one ordered pair from each equivalence class,
so there are only $q+1$ 
%
possible pairs that can occur. Since there are $q+1$ columns, it follows that
from each of these $q+1$ equivalence classes, exactly one will be chosen. 
Therefore, each row must contain exactly one $0$ and thus $M$ is of type $q+1$ standard form.

From the above discussion, we see that $M$ has the following structure:
\[
\left( \begin{array}{c c c c c c c c}
0 & 1 & 1 & 1 & \dots& 1& 1 \\
1 & 0 &  &  &  &  & \\
1 &  & 0 &  &  &  & \\
1 &  &  & 0 &  &  & \\
\vdots &  &  &  & \ddots&  & \\
1 &  &  &  &  & 0 & \\
1 &  &  &  &  &  &0
\end{array}\right) .
\]
Now consider the lower right $q$ by $q$ submatrix $M'$ of $M$.
There is exactly one occurrence of each element of ${\eff_q}^*$ in each column of $M'$. Now, compute the sum of all the rows in this matrix. Recall that the sum of the elements of a finite field $\eff_q$ is equal to $0$, provided that $q > 2$. Therefore, regardless of the configuration of the remaining entries, the sum of the last $q$ rows of $M$ is the all-zero vector. Therefore, the matrix $M$ is singular, which contradicts its being an AONT.
\end{proof}

\begin{Remark}
\label{nonexist-rem}
In \cite[Example 16]{DES}, it is shown that a linear $(2,3,2)$-AONT does not exist.
This covers the exception $q=2$ in Theorem \ref{nonexistence}.
\end{Remark}

\subsection{Linear $(2,q,q)$-AONT}
\label{2qq.sec}

We next obtain some structural conditions for linear $(2,q,q)$-AONT in standard form.

\begin{Lemma} Suppose $M$ is a matrix for a linear $(2,q,q)$-AONT in standard form.
Then $M$ is of type $q$ or type $q-1$.
\end{Lemma}

\begin{proof}
Suppose that $M$ is of type $\mu$ standard form, where $\mu \leq q-2$. Then the last two rows of $M$ contain 
no zeroes. We proceed as in the proof of Theorem \ref{nonexistence}. The $q$ ordered pairs in
the last two rows must all be from different equivalence classes. However, there are only $q-1$
equivalence classes that do not contain a $0$, so we have a contradiction.
\end{proof}

Therefore the standard form of a  linear $(2,q,q)$-AONT 
looks like \[ M= 
\left( \begin{array}{c c c c c c c c}
0 & 1 & 1 & 1 & \dots& 1& 1 \\
1 & 0 &  &  &  &  & \\
1 &  & 0 &  &  &  & \\
1 &  &  & 0 &  &  & \\
\vdots &  &  &  & \ddots&  & \\
1 &  &  &  &  & 0 & \\
1 &  &  &  &  &  & \chi
\end{array}\right) ,
\]
where $\chi=0$ iff $M$ is of type $q$ and $\chi\neq 0$ iff $M$ is of type $q-1$.

For the rest of this section, we will focus on linear $(2,q,q)$-AONT in type $q$ standard form.
Suppose $M$ is a matrix for such an AONT.
Define a linear ordering on the elements in the alphabet $\eff_q$.
If $M$ also has the additional property that the entries
in columns $3, \dots , q$  of row $2$ are in increasing order (with respect to this linear order), 
then we say that
$M$ is \emph{reduced}. So the term ``reduced'' means that $M$ is a linear $(2,q,q)$-AONT
that satisfies the following additional properties:
\begin{itemize}
\item the diagonal of $M$ consists of zeroes,
\item the remaining entries in the first row and first column of $M$ are ones, and
\item the entries
in columns $3, \dots , q$  of row $2$ of $M$ are in increasing order.
\end{itemize}

\begin{Lemma}
\label{reduced} Suppose $M$ is a matrix for a linear $(2,q,q)$-AONT in type $q$ standard form.
Then we can permute the rows and columns of $M$ to obtain a reduced matrix $M'$.
\end{Lemma}

\begin{proof}
Let $\pi$ be the permutation of $3, \dots , q$, which, when applied to the columns of
$M$, results in the entries
in columns $3, \dots , q$ of row $2$ being in increasing order. Call this matrix
$M^{\pi}$.
Now, apply the same permutation $\pi$ to the rows of $M^{\pi}$ to construct the desired
reduced matrix $M'$.
\end{proof}

\subsection{Some Computer Searches for Small Linear $(2,q,q)$-AONT}
\label{computer.sec}

We have performed exhaustive searches for reduced $(2,q,q)$-AONT (which are by definition
linear AONT in type $q$ standard form) 
for all prime powers $q \leq 11$. The results are found in Table \ref{tab1}.
(The notion of ``equivalence'' will be discussed in  Section \ref{equivalence.sec}.)

One perhaps surprising outcome of our computer searches is that there are 
no linear $(2,q,q)$-AONT in type $q$ standard form for $q = 8,9$
(however,  it is easy to find examples of linear $(2,q-1,q)$-AONT
for $q = 8,9$).
We also performed an exhaustive search for linear $(2,q,q)$-AONT in type $q-1$ standard form
for $q \leq 9$, and we did not find any examples. 

\begin{table}
\caption{Number of reduced and inequivalent linear $(2,q,q)$-AONT, for prime powers $q\le 11$}
\label{tab1}
\begin{center}
\begin{tabular}{|c|c|c|}\hline
$q$ & reduced $(2,q,q)$-AONT & inequivalent $(2,q,q)$-AONT\\ \hline
$3$ & $2$ & $1$ \\ \hline
$4$ & $3$ & $2$\\ \hline
$5$ & $38$ &  $5$\\ \hline
$7$ & $13$  & $1$ \\ \hline
$8$ & $0$  & $0$ \\ \hline
$9$ & $0$  & $0$ \\ \hline
$11$ & $21$  & $1$\\ \hline
\end{tabular}
\end{center}
\end{table}

For the prime orders $3,5,7,11$, it turns out that there
exists a reduced  $(2,q,q)$-AONT having a very interesting structure, which we define now.
Let $M$ be a matrix for a  reduced  $(2,q,q)$-AONT.
Let $\tau \in \eff_q$. We say that $M$ is \emph{$\tau$-skew-symmetric} if,
for any pair of cells $(i,j)$ and $(j,i)$ of $M$, where $2 \leq i,j \leq q$ and $i \neq j$,
it holds that $m_{ij} + m_{ji} = \tau$. Notice that this property implies that the matrix
$M$ contains no entries equal to $\tau$, since the only zero entries are on the diagonal.
Another way to define the $\tau$-skew-symmetric property is to say that
$M_1 + {M_1}^T = \tau (J-I)$, where $M_1$ is formed from $M$ by deleting the first row and column,
$J$ is the all-ones matrix and $I$ is the identity matrix.

Our computer searches show that there is a $(q-1)$-skew-symmetric reduced  $(2,q,q)$-AONT for 
$q= 3,5,7,11$, as well as $\tau$-skew-symmetric examples with
various other values of $\tau$.

\begin{Example}
\label{E1}
A $2$-skew-symmetric reduced linear $(2,3,3)$-AONT:
\[
\left(
\begin{array}{ccc}
0 & 1 & 1 \\
1 & 0 & 1\\
1 & 1 & 0
\end{array}
\right) .
\]
\end{Example}

\begin{Example}
\label{E2}
A linear $(2,4,4)$-AONT, defined over the finite field
$\eff_4 = \zed_2[x] / (x^2 + x + 1)$:
\[
\left(
\begin{array}{cccc}
0 & 1 & 1 & 1\\
1 & 0 & 1 & x\\
1 & x & 0 & x+1\\
1 & 1 & x & 0
\end{array}
\right) .
\]
\end{Example}

\begin{Example}
\label{E3}
A $4$-skew-symmetric reduced linear $(2,5,5)$-AONT:
\[
  \left(\begin{array}{c c c c c} 
 0 & 1 & 1 & 1 & 1\\
 1 & 0 & 1 & 2 & 3\\
 1 & 3 & 0 & 1 & 2\\
 1 & 2 & 3 & 0 & 1\\
 1 & 1 & 2 & 3 & 0
 \end{array}\right)	
  \]
 \end{Example}

\begin{Example}
\label{E4}
A $6$-skew-symmetric reduced linear $(2,7,7)$-AONT:
\[
\left(
\begin{array}{ccccccc}
 0 & 1 & 1 & 1 & 1 & 1 & 1\\
 1 & 0 & 1 & 2 & 3 & 4 & 5\\
 1 & 5 & 0 & 3 & 4 & 2 & 1\\
 1 & 4 & 3 & 0 & 5 & 1 & 2\\
 1 & 3 & 2 & 1 & 0 & 5 & 4\\
 1 & 2 & 4 & 5 & 1 & 0 & 3\\
 1 & 1 & 5 & 4 & 2 & 3 & 0
\end{array}
\right) .
\]
\end{Example}

\begin{Example}\label{289}
A linear $(2,8,9)$-AONT, defined over the finite field
$\eff_9  = \zed_3[x] / (x^2 + 1)$:
\[
\left( \begin{array}{c c c c c c c c}
0 & 1 & 1 & 1 & 1 & 1 & 1 & 1 \\
1 & 0 & 1 & 2 & x & x+1 & x+2 & 2x \\
1 & 1 & 0 & 2x+1 & x+1 & x+2 & 2 & x \\
1 & 2x & x & 0 & x+2 & 2 & 2x+1 & x+1  \\
1 & x+2 & 2 & x & 0 & 1 & 2x & 2x+1 \\
1 & x+1 & x+2 & 2x & 2x+1 & 0 & 1 & 2 \\
1 & x & x+1 & 1 & 2 & 2x+1 & 0 & x+2 \\
1 & 2 & 2x+1 & x+1 & 1 & 2x & x & 0 
\end{array} \right)
\]
\end{Example}

\begin{Example}
\label{E5}
A $10$-skew-symmetric reduced  linear $(2,11,11)$-AONT:
\[
\left(\begin{array}{c c c c c c c c c c c}
 0 & 1 & 1 & 1 & 1 & 1 & 1 & 1 & 1 & 1 & 1\\
 1 & 0 & 1 & 2 & 3 & 4 & 5 & 6 & 7 & 8 & 9\\
 1 & 9 & 0 & 7 & 8 & 1 & 3 & 2 & 5 & 4 & 6\\
 1 & 8 & 3 & 0 & 2 & 5 & 6 & 1 & 9 & 7 & 4\\
 1 & 7 & 2 & 8 & 0 & 6 & 1 & 3 & 4 & 9 & 5\\
 1 & 6 & 9 & 5 & 4 & 0 & 8 & 7 & 3 & 1 & 2\\
 1 & 5 & 7 & 4 & 9 & 2 & 0 & 8 & 1 & 6 & 3\\
 1 & 4 & 8 & 9 & 7 & 3 & 2 & 0 & 6 & 5 & 1\\
 1 & 3 & 5 & 1 & 6 & 7 & 9 & 4 & 0 & 2 & 8\\
 1 & 2 & 6 & 3 & 1 & 9 & 4 & 5 & 8 & 0 & 7\\
 1 & 1 & 4 & 6 & 5 & 8 & 7 & 9 & 2 & 3 & 0
\end{array}\right)
\]

\end{Example}

\subsection{Equivalence of Linear AONT}
\label{equivalence.sec}

In this section, we discuss how to determine if
two linear AONT are ``equivalent''. We define this notion as follows.
Suppose $M$ and $M'$ are   linear $(t,s,q)$-AONT.
We say that $M$ and $M'$ are \emph{equivalent} if $M$ can be transformed
into $M'$ by performing a sequence of operations of the following type:
\begin{itemize}
\item row and column permutations,
\item multiplying a row or column by a nonzero constant, and 
\item transposing the matrix.
\end{itemize}

Here, we confine our attention to reduced $(2,q,q)$-AONT, as defined in
Section \ref{2qq.sec}. We already showed that any linear $(2,q,q)$-AONT
of type $q$ standard form is equivalent to a reduced $(2,q,q)$-AONT.
But it is possible that two reduced $(2,q,q)$-AONT could be equivalent.
We next describe a simple process to test for equivalence of reduced $(2,q,q)$-AONT.

The idea is to start with a specific reduced $(2,q,q)$-AONT, say $M$.
Given $M$, we can generate all the reduced $(2,q,q)$-AONT that are equivalent to $M$.
After doing this, it is a simple matter to take any other reduced $(2,q,q)$-AONT, say $M'$
and see if it occurs in the list of reduced $(2,q,q)$-AONT that are equivalent to $M$.

The algorithm presented in Figure \ref{equivalent.fig} generates all the 
reduced $(2,q,q)$-AONT that are equivalent to $M$. 
After executing the first five steps, we have a list of $q^2-q$ reduced $(2,q,q)$-AONT,
each of which is equivalent to
$M$ (this includes $M$ itself). After transposing the original matrix, we repeat the same five
steps, which  gives $q^2 - q$ additional
equivalent AONT. The result is a list of $2q^2 - 2q$ equivalent AONT, but of
course there could be duplications in the list.

\begin{figure}
\begin{center}
\framebox{
\begin{minipage}[c][2.25in][c]{5.5in} 
\begin{enumerate}
\item Pick two distinct rows $r_1,r_2$. Interchange rows $1$ and $r_1$ of $M$
and interchange rows $2$ and $r_2$ of $M$. Then
interchange columns $1$ and $r_1$ and interchange columns $2$ and $r_2$ of the resulting matrix.
\item Multiply columns $2, \dots , q$ by constants to get $(0 \: 1 \:1 \: \cdots \: 1)$ in the first row.
\item Multiply rows $2, \dots , q$ by constants to get  $(0 \: 1 \:1 \: \cdots \: 1)^T$ in the first column.
\item Permute columns $3, \dots , q$ so the entries in row $2$ in these columns
are in increasing order (there is a unique permutation $\pi$ that does this).
\item Apply the same permutation $\pi$ to rows $3, \dots , q$.
\item Transpose $M$ and apply the first five steps to the transposed matrix.
\end{enumerate}
\end{minipage}
}
\end{center}
\caption{Generating the reduced $(2,q,q)$-AONT that are equivalent to a given reduced $(2,q,q)$-AONT, $M$}
\label{equivalent.fig}
\end{figure}

We have used this algorithm to determine the number of inequivalent $(2,q,q)$-AONT
for prime powers $q \leq 11$. We started with all the reduced  $(2,q,q)$-AONT and then
we eliminated equivalent matrices using our algorithm as described above.
The results are presented in Table \ref{tab1}.
 

\subsection{Additional Results on Linear AONT}
\label{lin-gen}

\begin{Theorem}
\label{cofactor}
If there exists a linear $(t,s,q)$-AONT with $t < s$,
then there exists a linear $(t,s-1,q)$-AONT.
\end{Theorem}

\begin{proof}
Let $M$ be a matrix for a
linear $(t,s,q)$-AONT.
Consider all the $s$ possible $s-1$ by $s-1$ submatrices formed
by deleting the first column and a row of $m$. We claim that at
least one of these $s$ matrices is invertible. For, if they were 
all noninvertible, then $M$ would be noninvertible, by considering
the cofactor expansion with respect the first column of $M$. 
\end{proof}

Given a prime power $q$,
define \[\mathcal{S}(q) = \{ s: \text{there exists a linear } (2,s,q)\text{-AONT}  \}.\]
From Remark \ref{ssq.rem}, we have that $2 \in \mathcal{S}(q)$, so $\mathcal{S}(q) \neq \emptyset$.
Also, from Theorem \ref{nonexistence}, Remark \ref{nonexist-rem}
and Theorem \ref{cofactor}, there exists a maximum element in $\mathcal{S}(q)$,
which we will denote by $M(q)$.
%
In view of Theorem \ref{cofactor}, we know that 
a linear $(2,s,q)$-AONT exists for all $s$ such that $2 \leq s \leq M(q)$.
We summarize  upper and lower bounds on $M(q)$  in Table \ref{tab2}.

\begin{table}[tb]
\caption{Upper and Lower bounds on $M(q)$}
\label{tab2}
\begin{center}
\begin{tabular}{c|c}
bound & authority \\ \hline
$\lfloor q/2 \rfloor \leq M(q) \leq q$ for all prime powers $q$ & Theorem \ref{cauchy} and \ref{nonexistence} \\
$M(q) \geq q-1$ if $q = 2^n$ and $q-1$ is prime &  Theorem \ref{vandermonde}\\ \hline
$M(q) = q$ for $q = 3,4,5,7,11$ & Examples \ref{E1}--\ref{E4} and \ref{E5}\\
$M(8) \geq 7$ & Theorem \ref{vandermonde} \\
$M(9) \geq 8$ & Example \ref{289}
\end{tabular}
\end{center}
\end{table}

\bigskip 

We finish this section by showing that the existence of linear AONT imply the existence of certain 
linear resilient functions.  
We  present the definition of resilient functions given in \cite{GS}. 
Let $|X| =v$. An \emph{$(n,m,t,v)$-resilient function}
is a function $g : X^n \rightarrow X^m$ which has the property that, if any $t$ of the $n$ input values 
are fixed and the remaining $n-t$ input values are chosen independently and uniformly at random,
then every output $m$-tuple occurs with the same probability $1/v^m$.

Suppose $q$ is a prime power. A $(n,m,t,q)$-resilient function $f$ is \emph{linear} if 
$f(x) = xM^T$ for some $m$ by $n$ matrix $M$ defined over $\eff_q$.

\begin{Theorem}
\label{RF-code}
Suppose there is a linear $(t,s,q)$-AONT. Then there is a linear $(s,s-t,t,q)$-resilient function.
\end{Theorem}

\begin{proof}
Suppose that the $s$ by $s$ matrix $M$ over $\eff_q$ gives rise to a linear $(t,s,q)$-AONT.
Then, from Lemma \ref{linear}, every $t$ by $t$ submatrix of $M$ is invertible. Construct an
$s$ by $t$ matrix $M^*$ by deleting any $s-t$ rows of $M$. Clearly any $t$
columns of $M^*$ are linearly independent. Let $\mathcal{C}$ be the code generated 
by the rows of $M^*$ and let $\mathcal{C}'$ be the dual code (i.e., the orthogonal
complement of $\mathcal{C}$). It is well-known from basic coding theory 
(e.g., see \cite[Chapter 1, Theorem 10]{MacSl})
that the minimum distance of $\mathcal{C}'$ is at least $t+1$. Let $N$ be a generating
matrix for $\mathcal{C}'$. Then $N$ is an $s-t$ by $s$ matrix over $\eff_q$. 
Since $N$ generates a code having minimum distance at least $t+1$, the function
$f(x) = xN^T$ is a a (linear) $(s,s-t,t,q)$-resilient function 
(for a short proof of this fact, see \cite[Theorem 1]{StMa}).
\end{proof}

\section{New Results on General AONT}
\label{general.sec}

In this section, we present a few results on ``general'' AONT 
(i.e., results that hold for any AONT, linear or not).

\begin{Theorem}
\label{AONT-OA}
Suppose there is a $(t,s,v)$-AONT. Then there is an OA$(t,s,v)$.
\end{Theorem}

\begin{proof}
Suppose we represent an $(t,s,v)$-AONT by a $(v^s,2s,v)$-array denoted by $A$.
Let $R$ denote the rows of $A$ that contain a fixed $(s-t)$-tuple in the last
$s-t$ columns of $A$. Then $|R| = v^t$. Delete all the rows of $A$ not in $R$ and
delete the last $s$ columns of $A$ and call the resulting array $A'$. Within 
any $t$ columns of $A$, we see that every $t$-tuple of symbols occurs exactly once,
since the rows of $A'$ are determined by fixing $s-t$ outputs of the AONT.
But this says that $A'$ is an OA$(t,s,v)$.
\end{proof}

The following classical bound can be found in \cite{CD}.
\begin{Theorem}[Bush Bound]
If there is an OA$(t,s,v)$, then
\[ s \leq 
\begin{cases}
v+t-1 & \text{if $t=2$, or if $v$ is even and $3 \leq t \leq v$}\\
v+t-2 & \text{if $v$ is odd and $3 \leq t \leq v$}\\
t+1 & \text{if $t \geq v$.}
\end{cases}
\]
\end{Theorem}

\begin{Corollary}
\label{bound2}
If there is a $(2,s,v)$-AONT, then $s \leq v+1$.
\end{Corollary}

We recall that we proved in Theorem \ref{nonexistence} that $s \leq v$ if a linear $(2,s,v)$-AONT exists;
the above corollary establishes a slightly weaker result in a more general setting.

\begin{Corollary}
\label{bound3}
If there is a $(3,s,v)$-AONT, then $s \leq v+2$ if $v \geq 4$ is even,
and $s \leq v+1$ if $v \geq 3$ is odd.
\end{Corollary}



Lastly, we prove a generalization of Theorem \ref{RF-code} which shows that any AONT (linear or nonlinear) 
gives rise to a resilient function. This result is based on a characterization of 
resilient functions which says that they are equivalent to ``large sets'' of orthogonal arrays.
Suppose $\lambda = v^r$ for some integer $r$. A \emph{large set of OA$_{v^{r}}(t,n,v)$}
consists of $v^{n-r-t}$ distinct OA$_{v^{r}}(t,n,v)$, which together contain all $v^n$ possible
 $n$-tuples exactly once. 

We will make use of the following result of Stinson \cite{St-LOA}.

\begin{Theorem}
\cite[Theorem 2.1]{St-LOA}
\label{RF-LOA}
An $(n,m,t,v)$-resilient function is equivalent to a large set of OA$_{q^{n-m-t}}(t,n,v)$.
\end{Theorem}

\begin{Theorem}
\label{AONT-RF}
Suppose there is a $(t,s,v)$-AONT. Then there is an $(s,s-t,t,v)$-resilient function.
\end{Theorem}

\begin{proof}
We use the same  technique that was used in the proof of Theorem \ref{AONT-OA}.
Let $A$ be the $(v^2,2s,v)$-array representing the AONT.
For any $(s-t)$-tuple $\mathbf{x}$, let $R_{\mathbf{x}}$ be the rows of $A$ that contain
$\mathbf{x}$ in the last $s-t$ columns of $A$. Let $A'_{\mathbf{x}}$ denote the submatrix
of $A$ indexed by the columns in $R_{\mathbf{x}}$ and the first $s$ columns. 
Theorem \ref{AONT-OA} showed that $A'_{\mathbf{x}}$ is an OA$(t,s,v)$.

Now, consider all $v^{s-t}$ possible $(s-t)$-tuples $\mathbf{x}$. For each choice
of $\mathbf{x}$, we get an OA$(t,s,v)$. These $v^{s-t}$ orthogonal arrays together contain
all $v^s$ $s$-tuples, since the array $A$ is unbiased with respect to the first $s$ columns.
Thus we have a large set of OA$_{1}(t,s,v)$. Applying Theorem \ref{RF-LOA}, this large set of
OAs is equivalent to an $(s,s-t,t,v)$-resilient function (note that $m = s-t$ because
$v^{s-m-t} = 1$).
\end{proof}

\section{Summary and Open Problems}
\label{summary.sec}

In this paper, we have begun a study of $t$-all-or-nothing transforms over alphabets
of arbitrary size. There are many interesting open problems suggested
by the results in this paper. We list some of these now.

\begin{enumerate}
\item Are there  infinitely many primes $p$ for which there exist linear $(2,p,p)$-AONT?
\item Are there  infinitely many primes $p$ for which there exist skew-symmetric linear $(2,p,p)$-AONT?
\item Are there any prime powers $q = p^i > 4$ with $i \geq 2$ for which 
there exist linear $(2,q,q)$-AONT?
\item As mentioned in Section \ref{computer.sec}, 
we performed exhaustive searches for linear $(2,q,q)$-AONT in type $q-1$ standard form, for 
all primes and prime powers $q \leq 9$, and found that no such AONT exist. We ask if there
exists any linear $(2,q,q)$-AONT in type $q-1$ standard form.
\item For $p = 3,5$, there are easily constructed examples of \emph{symmetric} linear $(2,p,p)$-AONT
in standard form (where ``symmetric'' means that $M = M^T$). 
But there are no symmetric examples for $p = 7$ or $11$. We ask if there
exists any symmetric linear $(2,p,p)$-AONT in standard form for a prime $p > 5$.
\item Theorem
\ref{cofactor} showed that a linear $(t,s-1,q)$-AONT exists whenever 
 a linear $(t,s,q)$-AONT exists. Does an analogous result hold for arbitrary
 (linear or nonlinear) AONT?
\item We proved in Theorem \ref{nonexistence} that, if a  linear $(2,s,q)$-AONT exists,  then $s \leq q$.
On the other hand,  for 
arbitrary (linear or nonlinear) $(2,s,v)$-AONT, 
we were only able to  show that $s \leq v+1$ (Corollary \ref{bound2}).
Can this second bound be strengthened to $s \leq v$, analogous to the linear case?
\item In the case $t=3$, we have one existence result
(Theorem \ref{cauchy}) and one necessary condition (Corollary \ref{bound3}).
What additional results can be proven about existence or nonexistence of $(3,s,v)$-AONT?
\end{enumerate}

\section*{Acknowledgements}

This work benefitted from the use of the CrySP RIPPLE Facility at
the University of Waterloo.

\end{document}